\documentclass[12pt]{article}
 \usepackage{alltt}
\usepackage{latexsym}
\usepackage{graphicx}
\usepackage{latexsym}
\usepackage{graphicx}
\usepackage{amssymb}
\usepackage{amsmath}
\usepackage{rotating}

\newtheorem{lemma}{Lemma}

\newenvironment{proof}{{\bf Proof.}}{\hfill $\Box$ \bigskip}
\usepackage{stmaryrd} 

\title{Global Attractivity of the Equilibrium \\ of a Difference Equation: \\
{\large An Elementary Proof Assisted by Computer Algebra System}}
\author{Orlando Merino \\
Department of Mathematics \\
University of Rhode Island\\
Kingston, RI 02881\\
{\it merino@math.uri.edu}
\date{ }
}

\begin{document}
\maketitle
\abstract{
Let $p$ and $q$ be arbitrary positive numbers.  
It is shown that if $q < p$, then all solutions to the difference equation
\begin{equation*}
\tag{E}
x_{n+1} = \frac{p+q\, x_n}{1+x_{n-1}}, \quad n=0,1,2,\ldots, \quad x_{-1}>0,\  x_0>0
\end{equation*}
converge to the positive equilibrium
$\overline{x} = \frac{1}{2}\left( q-1 + \sqrt{(q-1)^2 + 4 \,p}\right)$.
\medskip

The above result, taken together with the 1993 result of Koci\'c and Ladas
for equation (E) with $q \geq p$, gives global attractivity
of the positive equilibrium of (E) for all positive values of the parameters,
thus completing the proof of a conjecture of Ladas.
}
\vspace{1 in}


\vfill

\newpage
\section{Introduction and Statement of Main Theorem}
The difference equation
\begin{equation*}
\tag{E}
x_{n+1} = \frac{p+q\, x_n}{1+x_{n-1}}, \quad n=0,1,2,\ldots, \quad x_{-1}>0,\  x_0>0
\end{equation*}
has a unique equilibrium 
$\overline{x} = \frac{1}{2}\left( q-1 + \sqrt{(q-1)^2 + 4 \,p}\right)$.
In 1993, V. Koci\'c and G. Ladas  proved
in \cite{KoL93} (namely, part (e) of Theorem 3.4.3)
that for $q\geq p > 0$, all solutions to Eqn.(E)
converge to the equilibrium.
Since then the region of positive parameters $p,\ q$
for which global attractivity of the equilibrium holds 
has been extended in several publications:
V. Kocic, G. Ladas, and I. W. Rodrigues \cite{KLR}, 
C. H. Ou, H. S. Tang and W. Luo \cite{OTL00},
H. A. El-Morshedy \cite{EM03},
R. Nussbaum \cite{N07},  
V. Jim\'enez L\'opez  \cite{J08}.  
However, there is a large region of parameters 
that is not covered by the results obtained 
in those publications (see Figure \ref{fig: 0}).
Indeed, a well known conjecture of Ladas states  
{\it the positive equilibrium to {\rm (E)} is a global attractor
for all positive values of the parameters $p$ and $q$},
(for example, see \cite{Ladas95b}, \cite{KuL01}, \cite{KuLMR03}).
The authors in reference \cite{LZS} claim to have proved the conjecture,
but there is a mistake in their proof, see \cite{J08} for details on this.

The main result of this work, given below, 
establishes global attractivity of the equilibrium of (E)
for $q<p$, thus completing the proof of the conjecture.
\bigskip

\noindent {\bf Main Theorem } 
{\it
If $0<q<p$, then all solutions to the difference equation {\rm (E)}
converge to the positive equilibrium
$\overline{x}$.
}

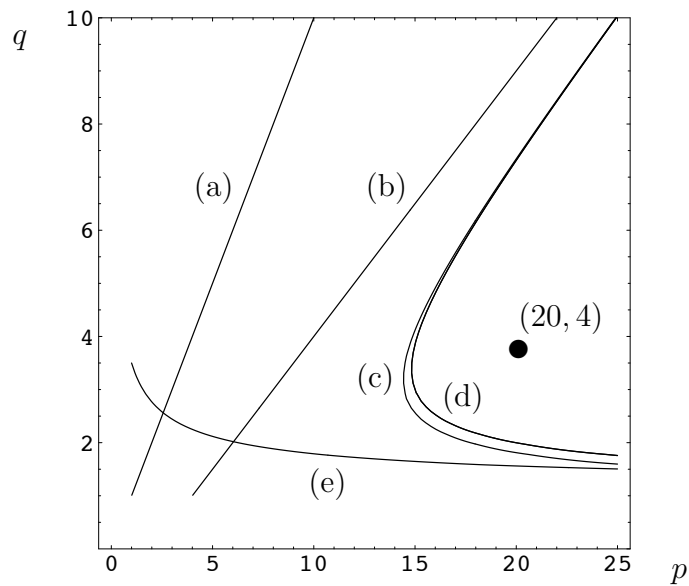
\begin{figure}[!ht]
\setlength{\unitlength}{1in}
\begin{center}
\begin{picture}(3,3)
\put(0,0){\includegraphics[width=3in]{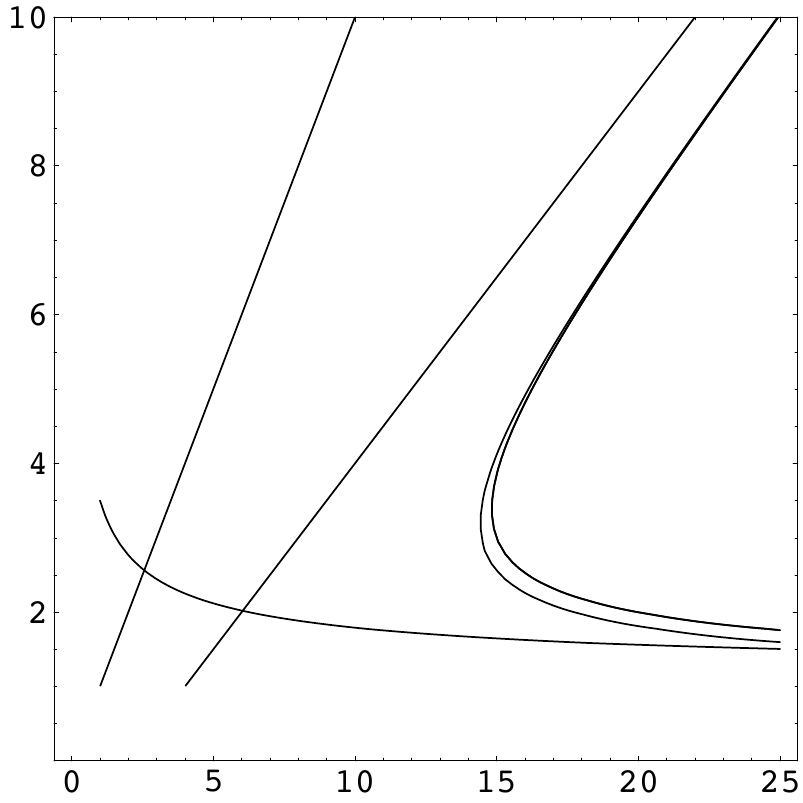}}
\put(3.2,0){$p$}
\put(-0.25,2.8){$q$}
\put(0.7,2){(a)}
\put(1.6,2){(b)}
\put(1.55,1){(c)}
\put(2,0.9){(d)}
\put(1.3,0.45){(e)}
\put(2.4,1.2){\circle*{0.1}}
\put(2.4,1.32){$(20,4)$}
\end{picture}
\end{center}
\label{fig: 0}
\caption{Boundaries of regions in parameter space.
Key: (a) $q \geq p$ (Koci\'c-Ladas, 1993), 
(b) $2(q+1) \geq p$ (Koci\'c-Ladas-Rodrigues, 1993),
(c) $ 2(q^3-q^2+q+\sqrt{q^4-1}-1)/(q-1)^2\geq p$ (Ou-Tang-Luo, 2000),
(d) $\frac{1}{2} \left(q+\sqrt{(q-1)^2+4
   p}-1\right)\leq \frac{q^2+1}{q-1}$ (El-Morshedy, 2003, and Nussbaum, 2007)
(e) $4\, p\, (q-1)^2\geq 25$ Jim\'enez L\'opez (2008).
It is interesting to note that Nussbaum obtained El-Morshedy's bound
expressed in different form.  
Also, Jim\'enez L\'opez' region  is an improvement
on regions obtained previously by 
Ou-Tang-Luo, El-Morshedy, and Nussbaum
for $q$ near $1$ and $p\geq 112.1$ (not shown in the figure).
The point $(p,q) = (20,4)$ is in the exterior of all regions
bounded by the different curves.
}
\end{figure}

\newpage

The study of Eq.(E) is facilitated by the introduction of a suitable 
change  of variables, and thus we begin with it.
The substitutions 
\begin{equation}
\label{eq: change of var}
p\, =\, \frac{\alpha}{A^2}, \quad q\, =\, \frac{1}{A}, \quad \mbox{ and }
\quad  \ x_\ell\,  =\,  q \, y_\ell \ ( \, \ell =-1,0,2,\ldots \, )
\end{equation}
 in Eq.(E) yield the equation
\begin{equation}
\label{eq: y2k2}
y_{n+1} = \frac{\alpha + y_n}{A+y_{n-1}}, \quad n=0,1,2,\ldots,
\quad y_{-1},\, y_0 \in (0,\infty)
\end{equation}
The equilibria $\overline{x}$ of Eq.(E) and $\overline{y}$ 
of Eq.(\ref{eq: y2k2}) satisfy
\begin{equation}
\overline{x} = q\, \overline{y}
\end{equation}
One may view Lyness' equation (see \cite{Ladas95}),
\begin{equation}
\label{eq: Lyness}
z_{n+1} = \frac{\tilde{\alpha} + z_n}{z_{n-1}}, \quad n=0,1,2,\ldots,
\quad z_{-1},\, z_0 \in (0,\infty)
\end{equation}
as a limit case of (\ref{eq: y2k2}) that results by setting $A=0$.
Lyness' equation is known  to posess an {\it invariant function} \cite{KUM}
i.e., a continuous, nonconstant real valued function $g(x,y)$ 
defined on $(0,\infty)\times (0,\infty)$ 
such that every solution $\{z_n\}$ to (\ref{eq: Lyness}) satisfies 
\begin{equation}
\label{eq: invariant 00}
g(z_{n+1},z_n) = g(z_0,z_{-1}), \quad n=0,1,2,\ldots
\end{equation}
A formula for $g(x,y)$ was found by Lyness \cite{Lyness} for $\tilde{\alpha}=1$,
and for positive $\tilde{\alpha}$ by Ladas \cite{Ladas95}:
\begin{equation}
\label{eq: invariant}
g(x,y) = \frac{(1+x)(1+y)(\tilde{\alpha}+x+y)}{x\, y}
\end{equation}
It is easy matter to show that $g(x,y)$ has a unique critical point on the positive quadrant,
namely $(\overline{z},\overline{z})$,
where $\overline{z}$ is the equilibrium of Eq.(\ref{eq: Lyness}).
Since $g(x,y)$ is large near the boundary of  the positive quadrant, $g(x,y)$ has 
a strict global minimum which is is attained at $(\overline{z},\overline{z})$
(see \cite{Zeeman}, \cite{Kulenovic}), i.e., 
\begin{equation}
\label{ineq: minimum}
g(\overline{z},\overline{z}) < g(x,y) \quad \mbox{for}\quad 
(x,y) \in (0,\infty)\times (0,\infty), \quad (x,y) \neq (\overline{z},\overline{z}).
\end{equation}
The function $g(x,y)$ plays a fundamental role in our proof.
See Figure \ref{fig: 1}.
\begin{figure}[!ht]
\centerline{\includegraphics[width=2in]{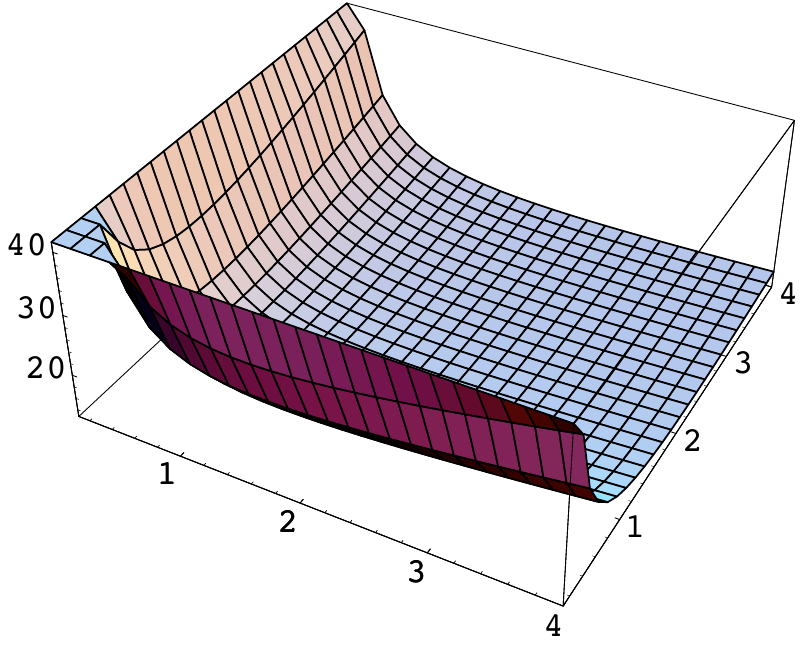}
\hspace{1 in} \includegraphics[width=2in]{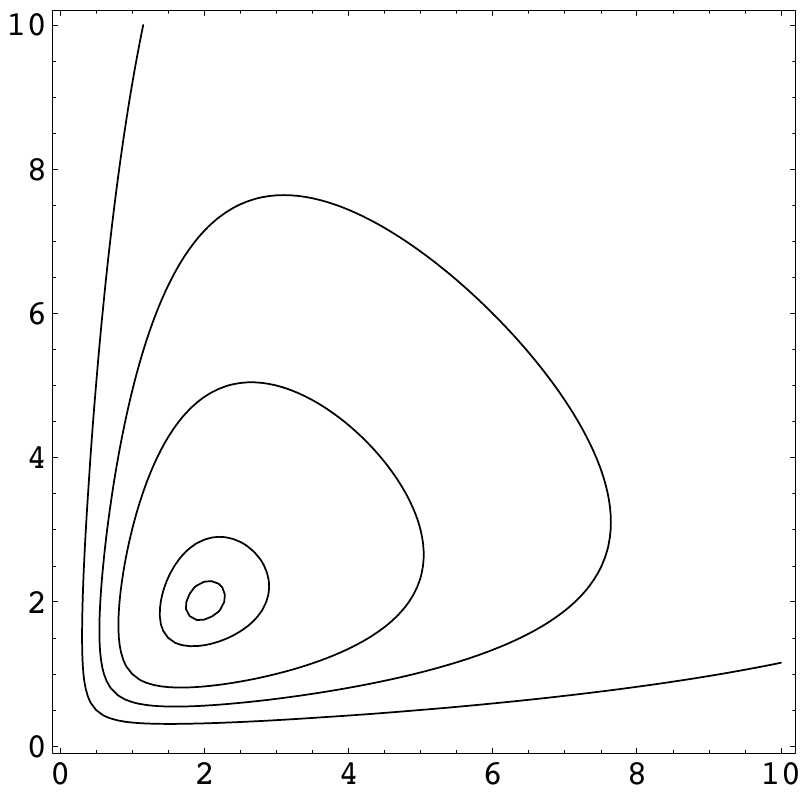}}
\caption{3D-Plot and contour plot of the invariant function $g(x,y)$ for $\tilde{\alpha}=2$.}
\label{fig: 1}
\end{figure}

In Section \ref{section: proofs} we prove several lemmas 
before giving the proof of the main theorem,
which is done
in order to simplify the exposition.
For practical reasons,
some of the calculations were performed with the computer algebra
system (or CAS) Mathematica \cite{Mathematica}.  
Code for such calculations, written in 
the Mathematica language,
is given in Section \ref{section: Mathematica}.

\bigskip

\section{Proofs}
\label{section: proofs}
For typographical convenience we will use the symbol $u$ to represent 
the equilibrium $\overline{y}$ of Eq.(\ref{eq: y2k2}).
By direct substitution of the equilibrium $u=\overline{y}$ into (\ref{eq: y2k2}) we obtain
\begin{equation}
\label{eq: alpha}
\alpha = u^2 + (A-1)\,u
\end{equation}
By (\ref{eq: alpha}), $\alpha > 0$ implies $u > 1-A$.
Using (\ref{eq: alpha}) to eliminate $\alpha$ from (\ref{eq: y2k2}) gives the following equation, equivalent to (E):
\begin{equation}
\label{eq: y2k2u}
y_{n+1} = \frac{u^2 +(A-1)\,u + y_n}{A+y_{n-1}}, \quad n=0,1,2,\ldots,
\quad y_{-1},\, y_0 \in (0,\infty)
\end{equation}
Therefore, to prove the main theorem it suffices to 
prove that all solutions of Eq.(\ref{eq: y2k2u})
converge to the equilibrium $u$, and this is what 
the rest of the proof is geared to do.

The following statement is crucial for the proof of the main theorem.
\begin{lemma}
\label{lemma: ybar > 1}
$u > 1$ if and only if $q < p$.
\end{lemma}
\begin{proof}
Since $\overline{x} = q\, \overline{y} = q\, u$, 
we have $u>1$ if and only if $\overline{x} > q$, which holds if and only if 
$\frac{1}{2}\left( q-1 + \sqrt{(q-1)^2 + 4 \,p}\right) > q$.
After an elementary simplification, the latter inequality 
can be rewritten as $q < p$.
\end{proof}

The hypothesis of the main theorem states $q<p$.
In view of Lemma \ref{lemma: ybar > 1} we
assume in the rest of the exposition that 
\begin{equation}
\label{assumption u>1}
u \ > \ 1
\end{equation}
For the equilibrium $u$ of Eq.(\ref{eq: y2k2u}) to be also the the positive equilibrium of Lyness' equation
(\ref{eq: Lyness}) one must have
\begin{equation}
u = \frac{\tilde{\alpha}+u}{u}
\end{equation}
that is,
\begin{equation}
\label{eq: tilde alpha}
\tilde{\alpha} = u^2 - u
\end{equation}
For the value of $\tilde{\alpha}$ found in 
(\ref{eq: tilde alpha}) the invariant function (\ref{eq: invariant}) becomes 
\begin{equation}
\label{eq: invariantu}
g(x,y) = \frac{(1+x)(1+y)(u^2-u+x+y)}{x\, y}
\end{equation}
One would hope that $g(x,y)$ turns out to be a {\it Lyapunov function} 
for Eq.(\ref{eq: y2k2u}) (see \cite{Kulenovic}). 
We shall see in the proof of Lemma \ref{lemma: Q2Q4} 
that this is not the case. 
Nevertheless we will be able to use the function $g(x,y)$  to complete the proof
of the main theorem.
First we need some elementary properties of the sublevel sets
$$
S(c) := \{ (s,t) : g(s,t) \leq c \}\, , \quad c>0
$$

We denote with $Q_\ell(u,u)$, $\ell=1,2,3,4$ the four regions
$$
\begin{array}{rcl}
Q_1(u,u) & := & \{ (x,y) \in (0,\infty)\times(0,\infty)\, : \ u \leq x,\ u\leq y \ \} \\  \\ 
Q_2(u,u) & := & \{ (x,y) \in (0,\infty)\times(0,\infty)\, : \ x \leq u,\ u\leq y \ \} \\ \\ 
Q_3(u,u) & := & \{ (x,y) \in (0,\infty)\times(0,\infty)\, : \ x \leq u,\ y\leq u \ \} \\ \\ 
Q_4(u,u) & := & \{ (x,y) \in (0,\infty)\times(0,\infty)\, : \ u \leq x,\ y\leq u \ \} \\ \\ 
\end{array}
$$
Let 
\begin{equation}
T(x,y) := \left(\, y\, ,\, \frac{u^2+(A-1)u + y}{A+x}\, \right), \quad (x,y) \in (0,\infty)\times(0,\infty)
\end{equation}
be the map associated to Eq. (\ref{eq: y2k2u})
(see \cite{KUM}).
\begin{lemma}
\label{lemma: Q2Q4}
If $(x,y) \in Q_2(u,u) \cup Q_4(u,u) \setminus \{ (u,u) \}$, 
then $g(T(x,y)) < g(x,y)$.
\end{lemma}
\begin{proof}
Set
\begin{equation}
\Delta_1(x,y) := g(x,y)-g(T(x,y)) 
\end{equation}
A calculation yields
\begin{equation}
\label{eq: Delta1}
\Delta_1(x,y)  =
\frac{A\,\left( u - u^2 + u\,x - 
      y \right) \,
    \left( u - A\,u - u^2 + A\,x + 
      x^2 - y \right) \,
    \left( 1 + y \right) }{x\,
    \left( A + x \right) \,y\,
    \left( -u + A\,u + u^2 + y
      \right)}
\end{equation}
By inspecting the factors in (\ref{eq: Delta1}) one can see that $\Delta_1(x,y)$ changes sign
on  the line $y=u - u^2 + u\,x$ and on the 
parabola $y=u - A\,u - u^2 + A\,x + x^2$.  Both curves are the graphs of strictly increasing 
functions of $x$ that intersect in the positive quadrant
$\{(x,y): x>0,\ y>0\}$ at a unique point, namely
 $(u,u)$. At the 
point $(u,u)$, the slope of the parabola, $2\,u+A$, is
larger than the slope of the line, $u$. See Figure \ref{fig: 2}.
Clearly $\Delta_1(x,y)$ becomes negative for points $(x,y)$ between both curves,
which, except for $(u,u)$,  are contained in the complement of $Q_2(u,u)\cup Q_4(u,u)$.
\begin{figure}[!ht]
\setlength{\unitlength}{1in}
\begin{center}
\begin{picture}(2,1.5)
\put(0,0){\includegraphics[width=2in]{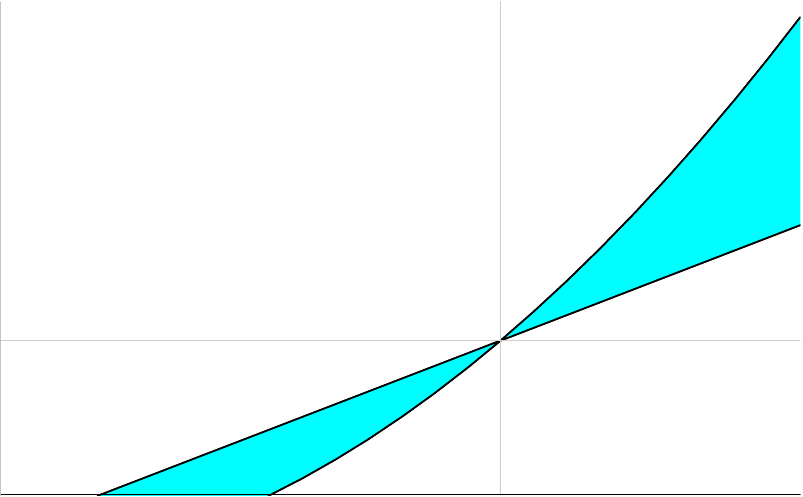}}
\put(1.2,-0.15){$u$}
\put(-0.2,0.34){$u$}
\put(1.9,-0.15){$x$}
\put(-0.2,1.2){$y$}
\end{picture}
\end{center}
\caption{$\Delta_1(x,y)<0$ for $(x,y)$ in the 
interior of the shaded region 
bounded by the  line $y=u - u^2 + u\,x$ and  the 
parabola $y=u - A\,u - u^2 + A\,x + x^2$.}
\label{fig: 2}
\end{figure}
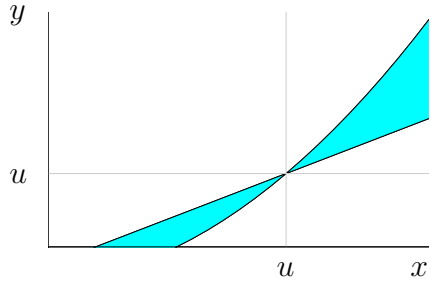
\end{proof}
\begin{lemma}
\label{lemma: Q1Q3}
If $(x,y) \in Q_1(u,u) \cup Q_3(u,u) \setminus \{ (u,u) \}$, 
then $g(T^2(x,y)) < g(x,y)$.
\end{lemma}
\begin{proof}
Set
\begin{equation}
\label{eq: Delta2 definition}
\Delta_2(x,y) := g(x,y)-g(T^2(x,y)) 
\end{equation}
A calculation yields
\begin{equation}
\label{eq: Delta2}
\Delta_2(x,y)  =
\frac{- A^2\,u^3  + 
  A^3\,u^3 + A^4\,u^3 - A^5\,u^3 + 
  \cdots 
  \mbox{\it 386 monomials}
  \cdots
  + 2\,A^2\,u\,x^2\,y^4 + 
  A\,u\,x^3\,y^4 + A\,y^5}
{x\,\left( A + x \right) \,y\,
  \left( A + y \right) \,
  \left( -u + A\,u + u^2 + y \right)
    \,\left( -u + A^2\,u + u^2 + 
    A\,u^2 - u\,x + A\,u\,x + 
    u^2\,x + y \right)}
\end{equation}
Since $u^2-u>0$ by (\ref{assumption u>1}), 
the denominator of $\Delta_2(x,y)$ is positive.
We now proceed to prove that the numerator of $\Delta_2(x,y)$,
denoted by $Num$,
is positive on
$Q_1(u,u)\setminus\{(u,u)\}$. Begin by changing variables:
$x = x_0+u$, $y=y_0+u$, where $x_0\geq0$, $y_0\geq0$ to get
\begin{equation}
\label{ineq: num}
Num = A\,u^4\,{x_0}^2 + 
  A\,u^5\,x_0^2 + 
  2\,A\,u^6\,x_0^2 + 
  2\,A\,u^7\,x_0^2 + 
  2\,A\,u^3\,x_0^3 - 
  A\,u^4\,x_0^3 +\cdots 
  \mbox{287 monomials}
\end{equation}
Since some monomials in the right-hand-side of 
(\ref{ineq: num})  have negative coefficient,
it is not obvious that $Num > 0$ for $x_0\geq 0$, $y_0\geq0$,
$(x_0,y_0)\neq (0,0)$.  The last step here consists in considering
subcases ($y_0>x_0$, $y_0<x_0$, and $y_0=x_0$).
Substituting $y_0 = x_0+k$ and $u=1+t$ for $k\geq 0$ and $t>0$,
we obtain 
\begin{equation}
\label{eq: num Delta2 on Q1}
Num = 
2\,A\,k^2 + 8\,A^2\,k^2 + \cdots \mbox{368 monomials} \cdots
 + 2\,A\,t\,x_0^7
\end{equation}
By using a computer algebra system one can show that
the  coefficients of the monomials that form the expression
in the right-hand-side of 
(\ref{eq: num Delta2 on Q1}) are all positive integers.
Thus $Num>0$ whenever 
(a) $x_0>0$, $y_0\geq x_0$.
It can be easily verified that  $Num>0$ in the other cases:
(b) $y_0>0$, $y_0\leq x_0$, 
(c) $x_0=0$, $y_0>0$, and 
(d) $x_0>0$, $y_0=0$.
We conclude that $\Delta_2(x,y)>0$ on 
$Q_1(u,u)\setminus \{(u,u)\}$.
See Section \ref{section: Mathematica} for details.

To verify that $Num>0$ on $Q_3(u,u)\setminus \{(u,u)\}$,
we consider first $Num$ on the interior of $Q_3(u,u)$, and  
introduce the algebraic transformation
$$
x=\frac{u\, w}{w+1}, \quad y = \frac{u \, v}{v+1},
\quad (w,v) \in (0,\infty)\times (0,\infty)
$$
that maps   
the open positive quadrant $\{(w,v) : w>0,\ v>0 \}$
one-to-one and onto the interior of 
$Q_3(u,u)$.
Then substitute $u=1+t$ with $t>0$, and consider 
subcases above, below, and on the diagonal just as we did 
before to conclude by direct inspection of coefficients that the transformed expression
for $Num$ is positive.

A similar strategy works for proving 
$\Delta_2(x,y) > 0$ on the open line segment  with endpoints 
$(u,0)$ and $(u,u)$, and on the open line segment
with endpoints $(0,u)$ and $(u,u)$.
See Section \ref{section: Mathematica} for details.
\end{proof}
\medskip

\noindent {\bf Proof of the Main Theorem.} 
Let 
$(\phi,\psi ) \in (0,\infty)\times(0,\infty)$.
Let $\{y_n\}_{n\geq-1}$ be the solution to (\ref{eq: y2k2u})
with initial condition $(y_{-1},y_0)=(\phi,\psi)$,
and let $\{T^n(\phi,\psi)\}_{n\geq 0}$ be the corresponding orbit of $T$.
Define 
\begin{equation}
\label{eq: hatc def}
\hat{c} := \liminf_n g(T^n(\phi,\psi))  
\end{equation}
Note that $\hat{c} < \infty$, which  can be shown by 
applying Lemma \ref{lemma: Q2Q4} and Lemma \ref{lemma: Q1Q3} repeatedly as needed
to obtain a nonincreasing subsequence of 
$\{g(T^n(\phi,\psi))\}_{n\geq 0}$ that is bounded below by $0$.
Let $\{g(T^{n_k}(\phi,\psi))\}_{k\geq 0}$ be a subsequence
convergent to $\hat{c}$.
Therefore there exists $c>0$ such that 
$$
g(T^{n_k}(\phi,\psi)) \leq c \quad \mbox{ for all $k\geq 0$,}
$$
that is, 
$$T^{n_k}(\phi,\psi)) \in S(c) := \{ (s,t) : g(s,t) \leq c \}\quad \mbox{for } k \geq 0
$$
The set $S(c)$ is closed by continuity of $g(x,y)$.
Boundedness of $S(c)$ follows from  
$$0< x,\, y <   
\frac{(1+x)(1+y)(u^2-u+x+y)}{x\, y} = g(x,y) = c ,
\quad \mbox{for} \quad (x,y) \in S(c)
$$
Thus $S(c)$ is compact, and 
there exists a convergent subsequence 
$\{ T^{n_{k_\ell}}(\phi,\psi))\}_\ell$ with limit $(\hat{x},\hat{y})$ (say).
Note that
\begin{equation}
\label{eq: hat c = g(xhat,yhat)}
\hat{c} = \lim_{\ell \rightarrow \infty} g(T^{n_{k_\ell}}(\phi,\psi))
= g(\hat{x},\hat{y})
\end{equation}
We claim that $(\hat{x},\hat{y}) = (u,u)$.
If not, then by 
Lemma \ref{lemma: Q2Q4} and Lemma \ref{lemma: Q1Q3},
\begin{equation}
\label{eq: cont ineq}
\min \{ g(T(\hat{x},\hat{y})),g(T^2(\hat{x},\hat{y}))\} < \hat{c}
\end{equation}
Let $\| \cdot \|$ denotes the euclidean norm.
By (\ref{eq: cont ineq}) and continuity, there exists $\delta > 0$ such that 
\begin{equation}
\label{eq: by continuity}
\| (s,t) - (\hat{x},\hat{y}) \| < \delta \implies 
\min \{ g(T(s,t)),g(T^2(s,t))\} < \hat{c}
\end{equation}
Choose $L\in \mathbb{N}$ large enough so that 
\begin{equation}
\label{eq: choose L}
\| T^{n_{k_L}}(\phi,\psi) - (\hat{x},\hat{y}) \| < \delta
\end{equation}
But then (\ref{eq: by continuity}) and (\ref{eq: choose L}) imply
\begin{equation}
\min \{ g(T^{n_{k_L}+1}(\phi,\psi)),g(T^{n_{k_L}+2}(s,t))\} < \hat{c}
\end{equation}
which contradicts the definition (\ref{eq: hatc def}) of $\hat{c}$.
We conclude $(\hat{x},\hat{y}) = (u,u)$.
From this and the definition of convergence of sequences we have that for every $\epsilon > 0$ 
there exists $L \in \mathbb{N}$ such that
$\| T^{n_{k_L}}(\phi,\psi) - (u,u) \| < \epsilon$.
Finally, since 
$$
\max \left\{ |y_{n_{k_L}-1} - u | , |y_{n_{k_L}} - u | \right\}
 \leq \| (y_{n_{k_L}-1} - u , y_{n_{k_L}} - u ) \|
 = \| T^{n_{k_L}}(\phi,\psi) - (u,u) \| 
$$
we have that for every $\epsilon > 0$ there exists 
$L \in \mathbb{N}$ such that
$|y_{n_{k_L}} - u |<\epsilon$ and 
$|y_{n_{k_L}-1} - u |<\epsilon$.
Since  $u$ is a locally asymptotically stable 
equilibrium for Eq.(\ref{eq: y2k2u})
(this is Lemma 3.4.1 in page 67 of \cite{KoL93}), it follows that 
$y_n \rightarrow u$.
This completes the proof of the Main Theorem.
\hfill $\Box$
\bigskip

\noindent {\bf Acknowledgment}
The author thanks M. Kulenovi\'c for 
helpful discussions and for suggesting
the change of variables (\ref{eq: change of var}) which established
a direct connection of the main difference equation
to Lyness' equation, and J. Montgomery for suggesting 
a simplification of the proof.
\bigskip


\newpage
\section{Appendix: Mathematica Code}
\label{section: Mathematica}

\begin{table}[!ht]
\centerline{\includegraphics[width=7in]{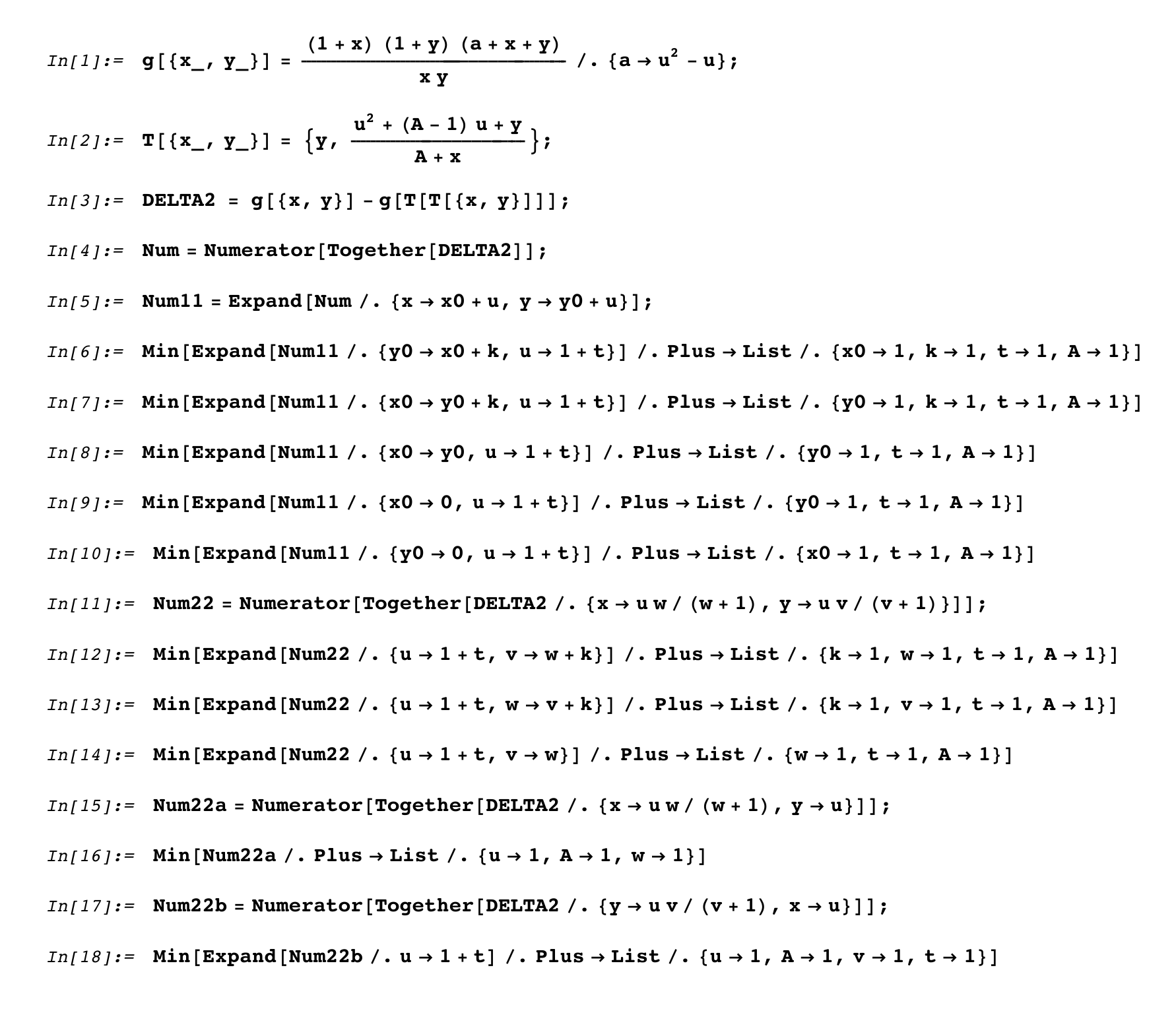}}
\caption{Mathematica code needed to do the calculations in Lemma \ref{lemma: Q1Q3}. 
  This input was tested on Mathematica Version 5.0
\cite{Mathematica}.
See Table \ref{table: two} for meaning of the input.}
\label{table: one}
\end{table}

\begin{table}[!ht]
\begin{center}
\begin{tabular}{|c|l|}
\hline & \\
Input \# & Description of input \\
& \\
\hline & \\
{\tt In[1]} & Define the invariant function for Lyness' equation. \\
& \\
{\tt In[2]} & Define the map associated to difference equation (\ref{eq: y2k2u}). \\
& \\
{\tt In[3]} & Define the expression $\Delta_2(x,y)$ given in (\ref{eq: Delta2 definition}). \\
& \\
{\tt In[4]} & Define word $Num$ to be numerator of expression $\Delta_2(x,y)$. \\
& \\
{\tt In[5]} & Replace $x$ and $y$ in $Num$ by $x_0+u$ and $y_0+u$. \\
& \\
{\tt In[6]} & Calculate smallest coefficient of monomials in $Num11$ for $y_0>x_0$. \\
& \\
{\tt In[7]} & Calculate smallest coefficient of monomials in $Num11$ for $x_0>y_0$.\\
& \\
{\tt In[8]} &  Calculate smallest coefficient of monomials in $Num11$ for $y_0=x_0$. \\
& \\
{\tt In[9]} &  Calculate smallest coefficient of monomials in $Num11$ on line $\{(0,y_0):\,y_0>0\}$. \\
& \\
{\tt In[10]} &  Calculate smallest coefficient of monomials in $Num11$ on line $\{(x_0,0):\,x_0>0\}$. \\
& \\
{\tt In[11]} & Numerator of map that takes $\mbox{int}\,Q_3(u,u)$ to the positive quadrant in the plane. \\
& \\
{\tt In[12]} & Calculate smallest coefficient of monomials in $Num22$ for $v>u$.  \\
& \\
{\tt In[13]} & Calculate smallest coefficient of monomials in $Num22$ for $v<u$.  \\
& \\
{\tt In[14]} & Calculate smallest coefficient of monomials in $Num22$ for $v=u$.  \\
& \\
{\tt In[15]} & Numerator of map that takes (open) line segment joining $(0,u)$, $(u,u)$ to $(0,\infty)$. \\
& \\
{\tt In[16]} &  Calculate smallest coefficient of monomials in $Num22a$. \\
& \\
{\tt In[17]} &  Numerator of map that takes (open) line segment joining $(u,0)$, $(u,u)$ to $(0,\infty)$. \\
& \\
{\tt In[18]} &  Calculate smallest coefficient of monomials in $Num22b$. \\
& \\
\hline 
\end{tabular}
\end{center}
\caption{Explanation for Mathematica input 
that appears in Table \ref{table: one}.}
\label{table: two}
\end{table}

\end{document}